\documentclass[10pt]{amsart}
\usepackage[margin=3.26cm]{geometry}

\usepackage[T1]{fontenc}
\usepackage[utf8]{inputenc}
\usepackage{lmodern}
\usepackage{amsmath,amssymb,amsthm,mathtools,microtype}

\usepackage{xcolor} 
\definecolor{citation}{rgb}{0.2,0.55,0.2}
\definecolor{formula}{rgb}{0.1,0.2,0.7}
\definecolor{urlc}{rgb}{0,0,0.45}
\usepackage[colorlinks=true,linkcolor=formula,citecolor=citation,urlcolor=urlc]{hyperref}

\numberwithin{equation}{section}
\newtheorem{theorem}{Theorem}[section]

\newtheorem{lemma}[theorem]{Lemma}

\theoremstyle{remark}

\newcommand{\R}{\mathbb R}
\newcommand{\HH}{\mathcal H}
\newcommand{\Ps}{P_s}
\title[A counterexample to the fractional Michael--Simon inequality]
{A counterexample to \\ the fractional Michael--Simon inequality}
\author[V. Julin]{Vesa Julin}
\address[V. Julin]{Department of Mathematics and Statistics, University of
Jyv\"askyl\"a, P.~\!O. Box 35, FI-40014 University of Jyv\"askyl\"a, Finland}
\email{vesa.julin@jyu.fi}

\author[G. Palatucci]{Giampiero Palatucci}
\address[G. Palatucci]{Department of Mathematical, Physical and Computer Sciences,
University of Parma, Parco Area delle Scienze 53/A, 43124 Parma, Italy}
\email{giampiero.palatucci@unipr.it}

\makeatletter\@namedef{subjclassname@2020}{\textup{2020} Mathematics Subject Classification}\makeatother
\thanks{{\it Funding.} V.~Julin was supported by the Academy of Finland grant 347550. G.~Palatucci has been supported by the INdAM project 2026
``Parabolic equations with drift: well posedness, gradient estimates and
boundary effects'' -- CUP\_E53C25002010001.}

\begin{document}
\begin{abstract}
We show that the  
fractional Michael-Simon inequality does not hold for
general smooth hypersurfaces and for the  range of $s \in (0,\frac{n-1}{n})$ in the case $p =1$. We construct a family of smooth hypersurfaces,  whose
fractional mean curvature is nonnegative everywhere, which violates the inequality in  the limit. 
The same construction yields a counterexample for every $1 \leq p < \infty$  for the range  $s < \frac{n-1}{np}$. 
\end{abstract}
\maketitle
\vspace{-5mm}

\section{Introduction}
The classical Michael-Simon-Allard inequality states that, for a
smooth $m$-dimensional submanifold $M$ without boundary in Euclidean
space and $1\le p<m$,
\begin{equation*}\tag{MS}\label{theMSA}
 \|u\|_{L^{p^*}(M)}
 \le C\Big(\|\nabla_Mu\|_{L^p(M)}
                 +\|H u\|_{L^p(M)}\Big),
 \qquad p^*=\frac{mp}{m-p},
\end{equation*}
for every compactly supported smooth function $u$ on $M$. Here $H$
is the mean curvature of $M$. This inequality was proved independently
by Allard~\cite{All72} and Michael and Simon~\cite{MS73}, and  the optimal constant in the case $p=1$ in  codimension at most two is due to Brendle~\cite{B 2021}. 

\smallskip

We consider the nonlocal counterpart for boundaries of sets in~$\R^n$, $n\ge2$. Let~$s\in(0,1)$, for a measurable set~$E\subset\R^n$, the fractional perimeter~$\Ps(E)$ is given by
\begin{equation*}
 \Ps(E)=\int_E\int_{\R^n\setminus E}
       \frac{\,{\rm d}y\,{\rm d}x}{|x-y|^{n+s}};
\end{equation*}
see~\cite{CRS10}.
If $E$ is bounded and open with $C^2$ boundary, its
fractional mean curvature~$H_s^E(x)$ at $x\in\partial E$ is
\begin{equation}\label{def_curvatura}
 H_s^E(x)= p.v.
 \int_{\R^n}\frac{\chi_{\R^n\setminus E}(y)-\chi_E(y)}{|x-y|^{n+s}}\,{\rm d}y,
\end{equation}
where {\it p.v.} stands for principal value.
As customary, we denote by~$P(E)$ the perimeter of~$E$.
For bounded Lipschitz domains it holds $P(E)=\HH^{n-1}(\partial E)$. We also write~$\omega_n=|B_1|$.
The natural fractional analogue of \eqref{theMSA} of order~$(s,p)$ for $0<s<1$ and $1\le p<\infty$ such that $sp<n-1$ reads as, for a constant~$C=C(n,s,p)$ independent of~$E$,
\begin{equation}\label{fractional_MSA}\tag{${\textrm{MS}}^{s}_{p}$}
 \|u\|_{L^{p^*_{s}}(\partial E)}
 \le C\left(
 [u]_{W^{s,p}(\partial E)}^p
 +\int_{\partial E}|H_s^E|^p|u|^p
      \,{\rm d}\HH^{n-1}\right)^{\frac{1}{p}},
\end{equation} 
for $u \in C^\infty(\partial E)$ when $\partial E$ is bounded and $u \in C_0^\infty(\partial E)$ in general see e.g. \cite{CCC23},  where  $p^*_{s}=\frac{(n-1)p}{n-1-sp}$\, and
$$
 [u]_{W^{s,p}(\partial E)}^p
 =\int_{\partial E}\int_{\partial E}
 \frac{|u(x)-u(y)|^p}{|x-y|^{n-1+sp}}
 \,{\rm d}\HH^{n-1}(y)\,{\rm d}\HH^{n-1}(x);
$$
see~\cite{DPV12}. 
Related fractional Sobolev inequalities are known for nonlocal minimal surfaces~(\cite[Proposition~5.2]{CC19}) and, more recently, for sets with vanishing fractional mean curvature~(\cite{Tho26}). Eventually,
Cabr\'e, Cozzi, and Csat\'o proved~\eqref{fractional_MSA} when~$E$ is convex  in an even more general formulation allowing the exponent in the curvature disagreeing with that of the fractional energy, in~\cite[Theorem~1.3]{CCC23} they use convexity in a
pointwise estimate that combines the fractional Sobolev kernel with~$H_\alpha^E$. Csat\'o and Roy subsequently extended this result to a fractional Caffarelli--Kohn--Nirenberg family on convex
hypersurfaces \cite[Theorem~1.1]{CR26}. 
A weaker fractional Michael--Simon inequality on convex hypersurfaces
under a subcritical fractional Willmore energy bound was obtained by Blatt, Giacomin, Schikorra, and Scheuer; see~\cite{BGSS25}. 
The extension of \eqref{fractional_MSA} beyond convexity is explicitly left open in~\cite[p.~190, after~Formula~(1.13)]{CCC23}, \cite[Section~1, before Theorem~1.1]{CR26} and~\cite[Section~1, after Formula~(1.7)]{Tho26}.

\smallskip

Observe that by considering \eqref{fractional_MSA} with the choice $u \equiv 1$, the inequality  becomes
\begin{equation*}\tag{${\widetilde{\mathrm MS}}^{s}_p$}\label{theMSA2}
 P(E)^{\frac{n-1-sp}{n-1}}
 \le C\int_{\partial E}|H_s^E|^p\,{\rm d}\HH^{n-1}.
\end{equation*} 
In particular, if \eqref{theMSA2} would hold then  a uniform bound on  $\|H_s^E\|_{L^p(\partial E)}$ would  imply uniform bound on the perimeter.  In this note we prove that the preceding  inequality, and therefore~\eqref{fractional_MSA}, fails for
general smooth sets whenever~$0<s<(n-1)/(np)$. The failure occurs even though~$H_s^E \ge 0 $ at every point of the boundary. This gives a negative answer to the extension problem in this range. It also precludes extending the whole family of Csat\'o and Roy~\cite[Theorem~1.1]{CR26} to all smooth boundaries with constants
independent of the hypersurface: the choice $a=1$, $\alpha=s$, $\gamma=0$ and $\tau=p^*_s$ in
their notation gives indeed~\eqref{fractional_MSA}.  

\smallskip

For the construction of our counterexample, we maximize the
fractional perimeter over a compact family of radial configurations,
allowing coincident radii. Each configuration admits an open representative whose boundary is a finite union of concentric
spheres. Homogeneity implies that the outermost boundary sphere of a maximizer has the prescribed radius. Radial variations give zero fractional mean curvature on every internal sphere, while a dilation computes the curvature on the outer sphere. The use of concentric annuli is inspired by the work of Cesaroni and Novaga~(\cite{CN20}). In Theorem~2.3 there, they balance the fractional mean curvature through the nonlocal interactions between concentric spheres to construct radial self-shrinkers. We use a similar type of interaction for a different purpose: we select the radii through a different variational problem to obtain the loss of coercivity between the classical perimeter and the fractional quantities.


\medskip
\section{The counterexample} 

Unless otherwise specified, all constants in this section depend only on~$n$, $s$, and~$p$.

\begin{theorem}\label{teo_controesempio}
Let $n\ge2$, $1\leq p < \infty$ and $0<s<(n-1)/(np)$. There exists a sequence of bounded open
sets $A_N\subset B_1$, $N = 4, 6, 8,  \dots$,  with
$C^\infty$ boundaries, such that
\begin{equation}\label{eq_contro}
 H_s^{A_N} \ge 0 \quad\hbox{on} \ \partial A_N,
 \qquad
 \frac{P(A_N)^{\frac{n-1-sp}{n-1}}}
 {\displaystyle\int_{\partial A_N}|H_s^{A_N}|^p\,{\rm d}\HH^{n-1}}
 \longrightarrow+\infty.
\end{equation}
 In particular, 
 the fractional Michael--Simon--Allard inequality cannot hold for all bounded
smooth open sets.
\end{theorem}

In order to prove the result above, we need some preliminary results.

\begin{lemma}\label{lem_interpol}
For every set $A \subset \R^n$ of finite measure and finite perimeter
one has
\begin{equation}\label{eq_interpol}
 \Ps(A)\le C|A|^{1-s}P(A)^s.
\end{equation}
If $\chi_{A_j}\to\chi_A$ in $L^1(\R^n)$ and
$\sup_jP(A_j)<\infty$, then $\Ps(A_j)\to\Ps(A)$.
\end{lemma}
\begin{proof}
The inequality~\eqref{eq_interpol}  is well-known, see e.g.  \cite[Corollary~4.3]{BLP14} or \cite[Formula~(1.4)]{PS17}. Indeed,   \cite[Proposition~4.2]{BLP14} gives the following general interpolation inequality,
\begin{equation}\label{eq_BV}
 [u]_{W^{s,1}}
 \le C\|u\|_{L^1}^{1-s}|Du|(\R^n)^s,
\end{equation}
for any $u\in BV(\R^n)\cap L^1(\R^n)$. Inequality~\eqref{eq_interpol} then  follows by applying~\eqref{eq_BV} to~$u = \chi_A$ and using  $[\chi_A]_{W^{s,1}}=2\Ps(A)$.

For the last assertion, lower semicontinuity of total variation gives
$P(A)<\infty$. Applying \eqref{eq_BV} to~$u_j=\chi_{A_j}-\chi_A$, whose total variations are uniformly bounded,
we obtain
\begin{equation*}
 2|\Ps(A_j)-\Ps(A)| \le  [\chi_{A_j}-\chi_A]_{W^{s,1}} \lesssim \|\chi_{A_j}-\chi_A\|^{1-s}_{L^1}  \longrightarrow 0. \qedhere
\end{equation*}
\end{proof}

\medskip
Fix $R\in (0,1]$ and an even integer $N\ge4$. Motivated by the concentric argument of Cesaroni and
Novaga in~\cite[Theorem~2.3]{CN20}, we define
$$
\mathbf J^{(N)}_R=\Big\{\mathbf r=(r_1,\ldots,r_N):0\le r_1\le\cdots\le r_N=R\Big\},
$$ 
and to any $\mathbf r\in{\mathbf J^{(N)}_R}$ we associate the set~$ A(\mathbf r)$ defined by
\begin{equation}\label{eq_annuli}
 A(\mathbf r)=\bigcup_{j=1}^{N/2}
 \left(B_{r_{2j}}\setminus\overline{B}_{r_{2j-1}}\right).
\end{equation}
Degenerate configurations are regarded as equivalence classes modulo
negligible  sets. Almost everywhere,
\begin{equation}\label{ball-rapp}
 \chi_{A(\mathbf r)}=\sum_{i=1}^N(-1)^i\chi_{B_{r_i}};
 \qquad
\text{and} \qquad  P(A(\mathbf r)) \le \sum_{i=1}^N P(B_{r_i})
   \le{n\omega_n}NR^{n-1}.
\end{equation}

\begin{lemma}\label{lemma_max}
For a fixed $N$, the function $\mathbf r\mapsto\Ps(A(\mathbf r))$ attains its
maximum on $\mathbf J_R^{(N)}$.
Every maximizer admits an open representative $A_{N,R}$ with
$C^\infty$ boundary consisting of at most $N$ concentric spheres,
including~$\partial B_R$, and satisfying
\begin{equation}\label{eq_zero}
 H_s^{A_{N,R}}=0 \quad {\text{on}}\ \partial A_{N,R}\cap B_R.
\end{equation}
\end{lemma}
\begin{proof}
By \eqref{ball-rapp}, for $\mathbf r,\mathbf t\in\mathbf J_R^{(N)}$,
$$
 \|\chi_{A(\mathbf r)}-\chi_{A(\mathbf t)}\|_{L^1}  \le \omega_n\sum_{i=1}^N|r_i^n-t_i^n|.
$$
The perimeters are uniformly bounded by \eqref{ball-rapp}.
Lemma~\ref{lem_interpol} therefore gives the continuity of~$\mathbf r\mapsto\Ps(A(\mathbf r))$.
Since~$\mathbf J_R^{(N)}$ is compact, a maximum is attained.

Let $\mathbf r$ be a maximizing vector and set~$\displaystyle A=A_{N,R}:=\operatorname{int}\bigl(\overline{A(\mathbf r)}\bigr)$.  
This representative coincides with~$A(\mathbf r)$ almost everywhere.

Notice that, since $B_R$ belongs to the family modulo negligible sets, it holds
$$
 \Ps(A)\ge\Ps(B_R)>0.
$$
The radial characteristic function 
has finitely many jumps, so
$$
 \partial A=\bigcup_{j=1}^m\partial B_{\rho_j},
 \qquad 0<\rho_1<\cdots<\rho_m\le R,
 \qquad m\le N.
$$
In particular, $\partial A$ is $C^\infty$.

We claim that $\rho_m=R$. Otherwise, choose
$1<\lambda<R/\rho_m$ and set
$$
 \widetilde r_i=\min\big\{\lambda r_i, R \big\},  \qquad 1\le i\le N.
$$
Then $\widetilde{\mathbf r}\in\mathbf J_R^{(N)}$ and, almost
everywhere,
$$
 A(\widetilde{\mathbf r})=(\lambda A)\cap B_R=\lambda A.
$$
Consequently,
$$
 \Ps(A(\widetilde{\mathbf r})) =  \Ps(\lambda A)=\lambda^{2n-(n+s)}\Ps(A) =\lambda^{n-s}\Ps(A)>\Ps(A),
$$
which is a contradiction in view of the maximality of~$A$.

\medskip
We are left to prove the equality \eqref{eq_zero}. To that aim we fix one of the internal spheres~$\partial B_{\rho_j}$,  choose a function $\psi \in C^1(\partial A)$ such that $\psi = 1$ on $\partial B_{\rho_j}$ and  $\psi = 0$ on $\partial A \setminus \partial B_{\rho_j}$. 
We construct a radial vector field $X \in C_0^1(U; \R^n)$, where $U$ is a neighborhood of $\partial B_{\rho_j}$ and which does not intersect the set $\partial A \setminus \partial B_{\rho_j}$,   such that $X \cdot \nu_A = \psi$ on $\partial A$. Consider the associated flow of diffeomorphisms  $\Phi_t$, i.e.,  $\frac{\partial}{\partial t} \Phi_t = X(\Phi_t)$ and $\Phi_0(x) = x$ for all $x \in \R^n$.  When $|t|$ is small the sets $\Phi_t(A)$ are admissible and  the first-variation Formula~(6.12) in~\cite[Theorem~6.1]{FFMMM15} implies 
\[
 0 =\left.\frac{{\rm d}}{{\rm d}t}
       \Ps(\Phi_t(A))\right|_{t=0} =\int_{\partial A}H_s^A(x)\,\psi(x) \,
       {\rm d}\HH^{n-1} = \int_{\partial B_{\rho_j}}H_s^A(x) \,
       {\rm d}\HH^{n-1} .
\]
By the radial symmetry $H_s^A$ is constant on~$\partial B_{\rho_j}$ and therefore  we deduce that $H_s^A = 0$ on~$\partial B_{\rho_j}$. This proves the claim. 
\end{proof}
\smallskip

We are now in the position to complete the construction of our counterexample.
\begin{proof}[Proof of Theorem~\ref{teo_controesempio}]
Let $N\ge4$ be even and let $A_{N,R}$ be the open representative
of a maximizer given by Lemma~\ref{lemma_max}.

\smallskip
\noindent\textit{Step 1: Estimates on the perimeters.}
We first show that there are constants~$c_*,C_*>0$, depending only on~$n$ and~$s$, such
that the family~\eqref{eq_annuli} contains a competitor~$A^*$ satisfying
\begin{equation*} 
 \Ps(A^*)\ge c_*N^sR^{n-s},
\end{equation*}
and every member $A$ of the family satisfies
\begin{equation*} 
 \Ps(A)\le C_*N^sR^{n-s},
\end{equation*}
so that, as a consequence, any maximizer $A_{N,R}$ in
Lemma~\ref{lemma_max} satisfies
\begin{equation}\label{eq_en_growth}
 c_*N^sR^{n-s} \le  \Ps(A_{N,R})\le C_*N^sR^{n-s}.
 \end{equation}

Set $\displaystyle h=\frac{1}{2}\frac{R}{N-1}$\, and choose
$$
\displaystyle r_i^*=\frac{R}{2}+(i-1)h,\qquad 1\le i\le N.
$$
Let
$$
 A^*=\bigcup_{j=1}^{N/2}
 \left(B_{r^*_{2j}}\setminus\overline{B}_{r^*_{2j-1}}\right).
$$
For $1\le j\le N/2-1$, set $q_j=r_{2j-1}^*$ and
$$
 X_j=B_{q_j+3h/4}\setminus\overline{B}_{q_j+h/4}
 \subset A^*.
$$
For $x=r\theta\in X_j$, with $\theta\in\mathbb S^{n-1} = \partial B_1$, define
$$
 \mathcal C(\theta)=
 \left\{\varphi\in\mathbb S^{n-1}:
 |\varphi-\theta|<\frac{h}{10R}\right\},
 \qquad
 \HH^{n-1}(\mathcal C(\theta))
 \ge c\left(\frac hR\right)^{n-1},
$$
and
$$
 Y_j(x)=
 \left\{t\varphi:
 q_j+\frac{5h}{4}<t<q_j+\frac{7h}{4},\
 \varphi\in\mathcal C(\theta)\right\}
 \subset\R^n\setminus A^*.
$$
Since all these radii lie between $R/2$ and $R$, polar coordinates give
\begin{align*}
 |X_j|
 &=n\omega_n\int_{q_j+h/4}^{q_j+3h/4}r^{n-1}\,{\rm d}r
 \ge cR^{n-1}h,\\
 |Y_j(x)|
 &=\HH^{n-1}(\mathcal C(\theta))
   \int_{q_j+5h/4}^{q_j+7h/4}t^{n-1}\,{\rm d}t
 \ge ch^n.
\end{align*}
Moreover, for $x=r\theta\in X_j$ and $y=t\varphi\in Y_j(x)$,
$$
 |x-y|
 = |(r-t)\theta+t(\theta-\varphi)|
 \le |r-t|+t|\theta-\varphi|
 \le\frac{3h}{2}+\frac h{10}\le2h.
$$
The sets $X_j$ are pairwise disjoint and $N/2-1\ge N/4$. Hence
\begin{align*}
 \Ps(A^*)
 &\ge\sum_{j=1}^{N/2-1}
 \int_{X_j}\int_{Y_j(x)}
 \frac{\,{\rm d}y\,{\rm d}x}{|x-y|^{n+s}}\\*[0.5ex]
 &\ge\frac{1}{(2h)^{n+s}}
 \sum_{j=1}^{N/2-1}\int_{X_j}|Y_j(x)|\,{\rm d}x\\*[1ex]
 &\ge cNR^{n-1}h^{1-s} \ \ge\  c_*N^sR^{n-s},
\end{align*}
as desired.

For  the upper bound, notice that, for any member $A$ of the family, \eqref{eq_interpol} and~\eqref{ball-rapp} give 
\begin{equation}\label{ps-r-est-unif}
 \Ps(A)
  \le C|A|^{1-s}P(A)^s 
  \le C(\omega_n R^n)^{1-s}
               ({n\omega_n}N R^{n-1})^s
 \le C_*N^s R^{n-s}.
 \end{equation}
Maximality  then yields the estimates in~\eqref{eq_en_growth}. 

Finally, such a maximizer~$A_{N,R}$ satisfies
 \begin{equation}  \label{eq_per-growth}
 cNR^{n-1} \le P(A_{N,R})\le{n\omega_n}NR^{n-1}.
\end{equation}
Indeed, by construction,~$c_*N^s R^{n-s}
 \le C R^{n(1-s)}P(A_{N,R})^s$, and thus taking the $s^{\textrm{th}}$-root gives the lower bound in
\eqref{eq_per-growth}; whereas the upper bound is already contained in~\eqref{ball-rapp}.

\smallskip
\noindent\textit{Step 2: Estimates on the nonlocal curvature.}
Write $A=A_{N,R}$. By radial symmetry and \eqref{eq_zero},
there is a constant $C_{N,R}$ such that
$$
 H_s^A(x)=
 \begin{cases}
  0, & x\in\partial A\setminus\partial B_R,\\*[0.5ex]
  C_{N,R}, & x\in\partial B_R.
 \end{cases}
$$

We choose the velocity vector field $X \in C_0^1(\R^n; \R^n)$ such that $X(x) = x$ in $B_{2R}$. Then for $x \in A$ and for $|t|$ sufficiently small  
it holds $\Phi_t(x)= e^t x$. Therefore, it holds 
\begin{equation}\label{flow2} 
\Ps(\Phi_t(A)) =  (e^t)^{n-s}\Ps(A) =e^{(n-s)t}\Ps(A).
\end{equation}
Since $X(x)\cdot \nu_A=R$ on $\partial B_R$ and $H_s^A = 0$ on $\partial A \setminus \partial B_R$, the first-variation
Formula~(6.12) in~\cite[Theorem~6.1]{FFMMM15} and~\eqref{flow2} imply
\begin{align}
 n\omega_n R^n C_{N,R}
 &=\int_{\partial A}H_s^A(x)\,X(x)\cdot \nu_A(x)\,
       {\rm d}\HH^{n-1}(x)\notag\\*[1ex]
 &=\left.\frac{{\rm d}}{{\rm d}t}
       \Ps(\Phi_t(A))\right|_{t=0}\notag\\*[1ex]
 &=(n-s)\Ps(A)>0.
 \label{eq_dilation}
\end{align}
Thus $C_{N,R}>0$ and $H_s^A\ge0$ on $\partial A$. Moreover,
for every $1\le p<\infty$ we have by the above 
\begin{align}
 \int_{\partial A}|H_s^A|^p\,{\rm d}\HH^{n-1}
 &=n\omega_n R^{n-1}C_{N,R}^p\notag\\*[0.5ex]
 &=\frac{(n-s)^p} {(n\omega_n)^{p-1}R^{n(p-1)+1}}\Ps(A)^p.
 \label{eq_compli}
\end{align}

\smallskip
\noindent\textit{Step 3: Completion of the proof.}
For each even $N\ge4$, take $A_N=A_{N,1}$ from
Lemma~\ref{lemma_max}. Then $A_N$ is open, is contained in $B_1$,
and has a $C^\infty$ boundary consisting of at most $N$ concentric
spheres. Steps~1 and~2 yield
\begin{align*}
 cN&\le P(A_N)\le CN,\\
 cN^{sp}&\le
 \int_{\partial A_N}|H_s^{A_N}|^p\,{\rm d}\HH^{n-1}
 \le CN^{sp},
\end{align*}
where the constants may also depend on $p$.
Then, up to relabeling the constants,
\begin{align*}
 \frac{P(A_N)^{\frac{n-1-sp}{n-1}}}
 {\displaystyle\int_{\partial A_N}|H_s^{A_N}|^p\,
       {\rm d}\HH^{n-1}}
 &\ge c\frac{N^{\frac{n-1-sp}{n-1}}}{N^{sp}}\ = \ cN^{1-\frac{nsp}{n-1}}\longrightarrow+\infty,
\end{align*}
for $0<s<(n-1)/(np)$, as claimed.
\end{proof}

%

\bigskip

\begin{thebibliography}{99}

\bibitem{All72}
W. K. Allard:
{On the first variation of a varifold}.
{\it Ann.~of~Math.} {\bf 95} (1972), 417--491.
\smallskip

\bibitem{BGSS25}
S. Blatt, G. Giacomin, A. Schikorra, J. Scheuer:
{A fractional Willmore-type energy functional-subcritical observations}.
In {\it MATRIX Book Series} {\bf 6}, Springer, Cham, 5--27, 2025.
\smallskip

\bibitem{BLP14}
L. Brasco, E. Lindgren, E. Parini: 
The fractional Cheeger problem. 
{\it Interfaces Free Bound.} {\bf 16} (2014), no. 3, 419--458.
\smallskip

\bibitem{B 2021}
S. Brendle:
{The isoperimetric inequality for a minimal submanifold in Euclidean space.} {\it J. Amer. Math. Soc.} {\bf 34} (2021),  595--603.

\smallskip
\bibitem{CC19}
X. Cabr\'e, M. Cozzi:
{A gradient estimate for nonlocal minimal graphs}.
{\it Duke Math. J.} {\bf 168} (2019), 775--848.
\smallskip

\bibitem{CCC23}
X. Cabr\'e, M. Cozzi,  G. Csat\'o:
{A fractional Michael--Simon Sobolev inequality on convex hypersurfaces}.
{\it Ann. Inst. H. Poincar\'e Anal. Non Lin\'eaire} {\bf 40} (2023), 185--214.
\smallskip


\bibitem{CRS10}
L. Caffarelli, J.-M. Roquejoffre,  O. Savin:
{Nonlocal minimal surfaces}.
{\it Comm. Pure Appl. Math.} {\bf 63} (2010), 1111--1144.
\smallskip

\bibitem{CN20}
A. Cesaroni, M. Novaga:
{Symmetric self-shrinkers for the fractional mean curvature flow}.
{\it J. Geom. Anal.}~{\bf 30}~(2020), 3698--3715.
\smallskip

\bibitem{CR26}
G. Csat\'o, P. Roy:
{A fractional Hardy-Sobolev inequality of Michael--Simon type on convex
hypersurfaces}.
{\it Math. Ann.} {\bf 395} (2026), Paper No.~14.
\smallskip

\bibitem{DPV12}
E. Di Nezza, G. Palatucci, E. Valdinoci:
{Hitchhiker's guide to the fractional Sobolev spaces}.
{\it Bull. Sci. math.}~{\bf 136} (2012), 521--573.
\smallskip

\bibitem{FFMMM15}
A. Figalli, N. Fusco, F. Maggi, V. Millot,  M. Morini:
{Isoperimetry and stability properties of balls with respect to nonlocal energies}.
{\it Comm. Math. Phys.} {\bf 336} (2015), 441--507.
\smallskip


\bibitem{MS73}
J.~\!H. Michael, L. M. Simon:
{Sobolev and mean-value inequalities on generalized submanifolds
of $\R^n$}.
{\it Comm. Pure Appl. Math.} {\bf 26} (1973), 361--379.
\smallskip

\bibitem{PS17}
A.~\!C. Ponce, D. Spector:
{A note on the fractional perimeter and interpolation}.
{\it C. R. Math. Acad. Sci. Paris} {\bf 355} (2017), 960--965.
\smallskip

\bibitem{Tho26}
J. Thompson:
{Density estimates and the fractional Sobolev inequality for
sets of zero $s$-mean curvature}.
{\it Calc. Var. Partial Differential Equations} {\bf 65} (2026), Article no.~2.

\end{thebibliography}
\end{document}